\documentclass[12pt,oneside,english]{amsart}
\usepackage[latin1]{inputenc}
\usepackage{amssymb}

\makeatletter
\newtheorem{theorem}{Theorem}
\newtheorem{lemma}{Lemma}

\newtheorem{corollary}{Corollary}

\usepackage{babel}

\makeatother
\begin{document}
\baselineskip=17pt
\title[ Two analogs of Thue-Morse sequence]
{Two analogs of Thue-Morse sequence}

\author{Vladimir Shevelev}
\address{Department of Mathematics \\Ben-Gurion University of the
 Negev\\Beer-Sheva 84105, Israel. e-mail:shevelev@bgu.ac.il}


\begin{abstract}
We introduce and study two analogs of one of the best known sequence in Mathematics
: Thue-Morse sequence. The first analog is concerned with the parity of number of runs of 1's
in the binary representation of nonnegative integers. The second one is connected with the parity of number of 1's in the representation of nonnegative integers in so-called negabinary (or in base $-2).$ We give for them some recurrent and structure formulas and prove that the second $(0,1)$-sequence is cube-free, while the first one is quint-free. Finally we consider several interesting unsolved problems.
\end{abstract}

\maketitle

\section{Introduction}

Let $T=\{t_n\}|_{n\geq0}$ be Thue-Morse sequence (or it is called also Prouhet-Thue
-Morse sequence, Allouche and Shallit [4]). It is defined as the parity of number
of $1's$ in binary representation of $n.$ $T$ is the sequence A010060 \cite{10}.
There are well known the following formulas for $T$ (\cite{4}):\newline
i) (A recurrent formula). $t_0=0,$ for $n>0,\enskip t_{2n}=t_n$ and
$t_{2n+1}=1-t_n.$\newline
ii) (Structure formula). Let $A_k$ denote the first $2^k$ terms; then $A_0 = 0$
and for $A_{k+1}, \enskip k >= 0,$ we have a concatenation  $A_{k+1} = A_k B_k,$ where $B_k$ is
obtained from $A_k$ by interchanging $0's$ and $1's;$\newline
iii) (A relation which is equivalent to ii)). For $0 <= k < 2^m, \enskip t_{2^m+k} =
1 - t_k;\newline$

Some much more general formulas one can find in author's article \cite{8}.
In this paper we introduce and study two analogs of $T:$\newline
1) Let $R=\{r_n\}|_{n\geq0}$ be the parity of number of runs of
 $1's$ in the binary representation of $n.$ $R$ is author's sequence A268411 in \cite{10}.
\newline
2) Let $G=\{g_n\}|_{n\geq0}$ be the parity of number of $1's$ in the negabinary
representation (or in base $-2$ of $n.$ (cf. \cite{11} and sequence A039724
\cite{10}; see also \cite[p.101]{6}, \cite[p.189]{7})).
$G$ is author's sequence A269027 in \cite{10}.\newline
\indent We say several words concerning the appearance of the sequences $G$ and $R.$
For the first time, the numerical base $-2$ was introduced by $V. Gr\ddot{u}nwald$
in 1885 (see talk in \cite{12} and references there). The author was surprised that during
130 years, nobody considered a natural analog of Thue-Morse sequence in the base $-2,$
and he decided to study this sequence. Unexpectedly, it turned out that
it has very interesting properties (although less canonical than Thue-Morse sequence).
Moreover, it definitely has astonishing (yet unproved) joint properties with Thue-Morse
sequence (see below our problems $C),D)).$ \newline
\indent Concerning sequence $R,$ it first appears in some-what exotic way. Let $u(n)$ be
characteristic $(0,1)$ function of a sequence $S=\{1=s_1<s_2<s_3<...\}.$ In the main
result of our recent paper \cite{9} there appears the following series
$f(x)=\sum_{i\geq2}(u(i)-u(i-1))x^i.$ It is easy to see that if to ignore zero
coefficients (when $u(i)=u(i-1)$), then other coefficients form alternative $(-1,1)$
sequence. The author decided to introduce a numerical system based on base $2$ with such an order
of digits. So he considered the "balanced binary" representation of $n$ which is
obtained from the binary representation of $n$ by replacing every $2^j$ by
$2^{j+1}-2^j.$ The system was named by the author "balanced" since the digital sum of every $n$
in this system equals 0. For example $7 = 4+2+1 =(8-4)+(4-2)+(2-1) = 8-1 =(1,0,0,-1)_b.$
The natural question: "how many pairs $1,-1$ are contained in the balanced binary
representation of $n?"$ is easily answered: this number equals the number of runs
of $1's$ in the binary representation of $n.$ This sequence modulo 2 is $R=A268411.$

\section{Main results}
In this paper we prove the following.
\begin{theorem}\label{t1}
The following recursion holds. $r_0=0, r_{2n}=r_n;$ for even $n,$ $r_{2n+1}=1-r_n;$
for odd $n,$  $r_{2n+1}=r_n.$
\end{theorem}
A useful corollary: $r_{4n}=r_n,$ $r_{4n+1} = 1 - r_n,$ $r_{4n+2} = r_{2n+1},$ $r_{4n+3} = r_{2n+1}.$
\begin{theorem}\label{t2}
Let $R_k$ denote the first $2^k$ terms of $R;$ then $R_1 = \{0,1\}$ and for $k >= 1,$
we have a concatenation  $R_{k+1} = R_k S_k,$ where $S_k$ is obtained from $R_k$
by complementing the first $2^{k-1}$  $0's$ and $1's$ and leaving the rest unchanged.
\end{theorem}
\begin{theorem}\label{t3}
The following recursion holds. $g_0=0, g_{4n}=g_n,$ $g_{4n+1} = 1 - g_n,$
$g_{4n+2} = 1 - g_{n+1},$ $g_{4n+3} = g_{n+1}.$
\end{theorem}
For the first time this statement was formulated by R. Israel in our sequence
$G=A269027$ \cite{10}.
\begin{theorem}\label{t4}
Let $G_k$ denote the first $2^k$ terms of $G;$ then $G_0={0}$ and for even
$k>=0,$ we have a concatenation $G_{k+1}= G_kF_k,$ where $F_k$ is obtained from $G_k$
 by complementing its $0's$ and $1's;$ for odd $k>=1,$ we have a concatenation
 $G_{k+1}= G_kH_k,$ where $H_k$ is  obtained from $G_k$ by complementing 
 \newpage
 its last
  $(2/3)(2^{k-1}-1)$ $0's$ and $1's.$
\end{theorem}

Recall that a sequence is cube-free if it contains no any subsequence of the form $XXX.$
For example, as is well-known, Thue-Morse sequence is cube-free.

\begin{theorem}\label{t5}
The sequence $G$ is cube-free.
\end{theorem}
\begin{theorem}\label{t6}
The sequence $R$ is quint-free, that is it contains no any subsequence of the form $XXXXX.$
\end{theorem}

\section{Proof of Theorem 1}
\begin{proof}
1)\enskip Trivially, $r_{2n}=r_n.$\newline
2) Let $n=2k, \enskip 2n+1=4k+1$ which ends on $00...01,$ where the number of zeros
 $\geq1,$ then the last 1 forms a new run of 1's. So, $r_{2n+1}=1-r_{4k}=1-r_{n}.$
 \newline
3) Let $n$ be odd such that $n-1=2^ml,$ where $l$ is odd, $m\geq1$ and
$2n+1=2^{m+1}l+3.$\newline
3a) Let $m=1.$ Then $4l$ ends on two zeros and the adding of 3 does not form a new run.
So, $r_{2n+1}=r_{4l+3}=r_{4l}=r_{2l}=r_{2l+1}=r_n.$\newline
3b) Let $m\geq2.$ Then $2^{m+1}l$ ends on $\geq3$ zeros and the adding of 3 forms a
new run. So, $r_{2n+1}=1-r_{2^{m+1}l}=1-r_{2^ml}=1-(1-r_{n})=r_n.$
\end{proof}
\section{Proof of Theorem 2}
\begin{proof}
It is easy to see that Theorem \ref{t2} is equivalent to the formula
\begin{equation}\label{1}
r_{n+2^k}=\begin{cases} 1-r_{n}, &\text{ $0\leq n\leq2^{k-1}-1$}\\r_n, &\text{
$2^{k-1}\leq n\leq2^k-1.$}\end{cases}
\end{equation}
In case when $n\in[2^{k-1},2^k-1]$ in the binary expansion of $n$
the maximal weight of 1 is $2^{k-1}.$ After addition of $2^k$ this new 1 continues
the previous run of 1's in which there is 1 of the weight $2^{k-1}.$ So, in this
case the number of runs of 1's does not change and $r_{n+2^k}=r_n.$
In opposite case when $n\in [0,2^{k-1}-1]$ after addition of $2^k$ this new 1 forms
a new run and the number of runs is increased on one, so $r_{n+2^k}=1-r_n.$
\end{proof}

\section{Proof of Theorem 3}
 In binary expansion of $n,$ we call \slshape even 1's \upshape the 1's with the weight
$2^{2k},\enskip k\geq 0,$ and other 1's we call \slshape odd 1's\upshape. In conversion from base 2 to base -2 an important role plays the parity of 1's in binary, since  only every odd
 1 with weight $2^{2k+1},\enskip k>=0,$ we should change by two 1's with weights
 $2^{2k+2},\enskip 2^{2k+1},$ which corresponds to the equality
 \newpage
 $$2^{2k+1}=(-2)^{2k+2}+(-2)^{2k+1}.$$
 For example $7=2^2+2+1=>2^2+2^2-2+1=2^3-2+1=2^4-2^3-2+1=11011_{-2}.$
\begin{proof}
1)Since multiplication $n$ by 4 does not change the parity of 1's, then, evidently,
 $g_{4n}=g_{n}.$\newline
2) Again evidently $g_{4n+1}=1-g_{4n}=1-g_n.$\newline
3) Note also that $g_{2n+1}=1-g_{2n}.$ Hence, $g_{4n+3}=g_{2(2n+1)+1}=1-g_{4n+2}.$
4) It is left to prove that $g_{4n+3}=g_{n+1}$ (then $g_{4n+2}=1-g_{n+1}).$
\newline
4a) Let $n$ be even $=2m.$ We should prove that $g_{8m+3}=g_{2m+1}.$ Note that $8m+3$
ends in the binary on $100...011,$ where the number of zeros$\geq1.$ Since $011$ in binary converts to $111_{-2},$ then $g_{8m+3}=1-g_{8m}=1-g_{2m}=1-(1-g_{2m+1})=g_{2m+1}.$\newline
4b) Let $n$ be odd and ends on even number of 1's. We need lemma.
\begin{lemma}\label{L1}
For $m\geq2,$
\begin{equation}\label{2}
2^m-1=\begin{cases}\enskip100...011_{-2},&\text{$m$ is even}\\1100...011_{-2}, &\text{$m$
is odd,}\end{cases}
\end{equation}
where in the 0's run we have $m-2$ zeros.

\end{lemma}

\begin{proof}
Let $m$ be even. Then we have
$$2^m-1= 2^{m-1}+2^{m-2}+...+2+1=$$
$$(2^m-2^{m-1})+2^{m-2}+(2^{m-2}-2^{m-3})+...+(16-8)+4+(4-2)+1=$$
$$2^m+(-2^{m-1}+2^{m-1})+(-2^{m-3}+2^{m-3})+...+(-8+8)-2+1=100...011_{-2} $$
such that zeros correspond to exponents $m-1,m-2,...,3,2,$ i.e., we have $m-2$
zeros. Now let $m$ be odd. Then $m-1$ is even and, using previous result, we have
$$2^m-1= 2^{m-1}+2^{m-2}+...+2+1=$$
$$2^{m-1}+2^{m-1}+(-2^{m-2}+2^{m-2})+(-2^{m-4}+2^{m-4})+...+(-8+8)-2+1=$$
$$2^{m+1}-2^m-2+1=1100...011_{-2}$$
with also $m-2$ zeros.
\end{proof}
\begin{corollary}\label{cor1}
\begin{equation}\label{3}
g_{2^m-1}=\begin{cases}0,&\text{$m=0$}\\1,&\text{$m=1$}\\1,&\text{$m\geq2$ is even}\\0, &\text{$m\geq3$
is odd.}\end{cases}
\end{equation}
\end{corollary}
\newpage
Let $n$ in the binary ends on even $m\geq2$ 1's. Then $4(n+1)$ ends on
$100...0$ with $m+2$ zeros and thus the end of $4(n+1)$ equals $100...0_{-2}$ with $m+2$ zeros.
On the other hand, $4n+3$ ends on $m+2\geq4$ 1's: $011...1,$ so, by (\ref{2}), the end
of $4n+3$ equals $10...011_{-2}$ with $m$ zeros. Since all the previous binary digits
for $4n+3$ and $4(n+1)$ are the same (indeed,
$4n+4-(4n+3)=1=10...0_2-01...1_2),$
then, continuing conversion from base 2 to base $-2,$ we obtain the equality
 $g_{4n+3}=g_{n+1}.$

4c) Finally, let $n$ be odd and ends on odd number $m$ of 1's. Consider in
more detail the end of $n.$ If $n$ ends on $001..1_2,$ then the proof does not differ
from the previous case, since $4(n+1)=...010...0_2$ with $m+2$ zeros, thus the end of $4(n+1)$ equals $110...0_{-2};$ on the other hand, $4n+3=...001...111_2$ and, by (\ref{2}) the end of
 $4n+3$ equals $110...011$ with $m$ zeros and we again conclude that $g_{4n+3}=
 g_{n+1}.$ \newline
\indent Now let $n$ ends on $01...101...1,$ where the first
run contains $k$ 1's while the second run contains odd $m$ 1's. Then $4n+3$ ends
on $01...101...111,$ while $n+1$ ends on $01...110...0,$ where the run of 1's
contains $k+1$ 1's which is followed by the run of $m$ 0's. Let us pass in two last
ends to base $-2.$ We need a lemma which is proved in the same way as Lemma \ref{L1}.
\begin{lemma}\label{L2}
For odd $m\geq3,$
\begin{equation}\label{4}
2^{m+k+1}-2^m-1=\begin{cases}\enskip 10...010...011_{-2},&\text{$k$ is even}\\110...
010...011_{-2}, &\text{$k$
is odd,}\end{cases}
\end{equation}
where the first (from the left to the right) 0's run has $k$ zeros, while the second 0's run has $m-2$ zeros;
\begin{equation}\label{5}
2(2^{k+1}-1)=\begin{cases}\enskip 10...010_{-2},&\text{$k$ is even}
\\110...010_{-2}, &\text{$k$
is odd,}\end{cases}
\end{equation}
where the first 0's run has $k$ zeros.
\end{lemma}
So, by (\ref{4}), for the end of $4n+3$ we have

 $$01...101...111_2=2^{(m+2)+k+1}-2^{m+2}-1=$$ 
 $$\begin{cases}\enskip 10...010...011_{-2},&\text{$k$
is even}\\110...010...011_{-2}...
010...011_{-2}, &\text{$k$
is odd,}\end{cases} $$
where the 0's runs have $k$ and $m$ zeros respectively.\newline
\indent For the corresponding end of $4(n+1)$ having $m+2$ zeros
at the end, by (\ref{5}), we have (since $m+1$ is even):
$$01...110...0_2=2^{m+1}\cdot 1...110_2=$$ $$2^{m+1}\cdot\begin{cases}\enskip 10...010_{-2},&\text{$k$ is even}\\110...010_{-2}, &\text{$k$
is odd}\end{cases}=$$
\newpage
$$\begin{cases}\enskip 10...010...0_{-2},&\text{$k$ is even}\\
110...010...0_{-2}, &\text{$k$ is odd}\end{cases}$$
where the 0's runs have $k$ and $m+2$ zeros respectively.\newline
Since all the previous binary digits for $4n+3$ and $4(n+1)$ are the same
$(4n+4-(4n+3)=1=1...110...0_2-1...101...1_2),$
then, continuing conversion from base 2 to base $-2,$ we obtain the equality $g_{4n+3}=g_{n+1}.$
\end{proof}

\section{Proof of Theorem 4}
\begin{proof}
It is easy to see that Theorem \ref{t4} is equivalent to the formula
\begin{equation}\label{6}
g_{2^k+m}=\begin{cases}1-g_m, &\text{$k \enskip is \enskip even\geq2$ and $2^{k-1}\leq m <2^k$}\\g_m,&\text{$k\enskip is\enskip odd\geq1$ and $0\leq m<2^k-\frac{2}{3}
(2^{k-1}-1)$}\\1-g_m, &\text{$k\enskip is\enskip odd\geq3$ and $2^k-\frac{2}{3}(2^{k-1}-1)\leq m<2^k.$}
\end{cases}
\end{equation}
1) Let $k$ be even$\geq2$ and $2^{k-1}\leq m <2^k.$ We use induction over $k.$
For $k=2,$ $2\leq m<4,$ (\ref{6}) is true: $g_{4+2}=1-g_{2},
\enskip g_{4+3}=1-g_{3}=0;$ also it is easy verify (\ref{6}) for $k=4.$ Suppose that
(\ref{6}) is true for $k-2.$ We write in binary ${1\vee m}$ instead of ${2^k+m}.$
\newline
1a) Let $m=4x.$ By the induction supposition $g_{1\vee x}=1-g_{x}.$ But also,
 by the condition, $g_{1\vee m}=g_{1\vee x}=1-g_{x}$ and $g_m=g_x.$ So
$g_{1\vee m}=1-g_m.$\newline
1b) Let $m=4x+1.$ By the induction supposition $g_{1\vee x}=1-g_{x}.$ But also,
 by the condition (since $g_{4n+1}=1-g_n)$ we have $g_{1\vee m}=1-g_{1\vee x}=
 g_{x}$ and $g_m=1-g_x.$ So $g_{1\vee m}=1-g_m.$\newline
1c) Let $m=4x-1.$ By the induction supposition $g_{1\vee x}=1-g_{x}.$ But also,
 by the condition (since $g_{4n-1}=g_{4(n-1)+3}=g_n)$ $g_{1\vee m}=g_{1\vee x}=
 1-g_{x}$ and $g_m=g_x.$ So $g_{1\vee m}=1-g_m.$ \newline
1d)  Let $m=4x-2.$ By the induction supposition $g_{1\vee x}=1-g_{x}.$ But also,
 by the condition (since $g_{4n-2}=g_{4(n-1)+2}=1-g_n)$ $g_{1\vee m}=1-g_{1\vee x}=
 g_{x}$ and $g_m=1-g_x.$ So $g_{1\vee m}=1-g_m.$ \newline
 The proof in the following two points is the same, except for the bases of induction.
 Therefore in the points 2),3) we give the bases of induction only.\newline
 2) Let $k$ be odd$\geq1$ and $0\leq m<2^k-\frac{2}{3}(2^{k-1}-1).$ For $k=1,$ we have
 $m=0,1$ and (\ref{6}) is true: $g_{2+0}=g_0=0$ and $g_{2+1}=g_1=1;$ for $k=3,$
 we have $m=0,1,2,3,4,5$ and (\ref{6}) is true: $g_{8+0}=g_0=0,$  $g_{8+1}=g_1=1,$
 $g_{8+2}=g_2=0,$ $g_{8+3}=g_3=1,$ $g_{8+4}=g_4=1,$ $g_{8+5}=g_5=0.$\newline
 3) Let $k$ be odd$\geq3$ and $\frac{2}{3}(2^{k-1}-1)\leq m<2^k.$ For $k=3,$ we have
 $m=6,7$ and (\ref{6}) is true: $g_{8+6}=1-g_6=0$ and $g_{8+7}=1-g_7=1.$ for $k=5,$
 we have $m=22,23,...,31$ and (\ref{6}) is true: $g_{32+22}=1-g_{22}=1,$ $g_{32+23}
 =1-g_{23}=0$ ... $g_{32+31}=1-g_{31}=1.$
\end{proof}
\newpage
\section{Proof of Theorem 5}
 \begin{proof} Suppose the sequence $G$ contains a subsequence of the form $XXX$ with the positions
 $$(4k,...,4k+s-1)(4k+s,...,4k+2s-1)(4k+2s,...,4k+3s-1),$$
 where $k\geq0, s\geq1$ are integers.
 We get a contradiction using the formulas of Theorem \ref{t3}. The cases when the first element of $X$ is $4k+1,2,3$ are considered in the same way.\newline
 1a) The case $s=1$ directly contradicts these formulas. Thus $G$ contains no three equal terms. \newline
 1b) $s=2.$ We have $g_{4k}=g_{4k+2}=g_{4k+4},$ then $g_{k}=1-g_{k+1}=g_{k+1},$ a contradiction.
 \newline
 1c) $s=3.$ We have $g_{4k}=g_{4k+3}=g_{4k+6}$ and $g_{4k+1}=g_{4k+4}=g_{4k+7},$ then  $g_{k}=g_{k+1}$ and $1-g_{k}=g_{k+1}=g_{k}.$ A contradiction.\newline
 1d) $s=4.$ We have $g_{4k}=g_{4k+4}=g_{4k+8},$ then $g_k=g_{k+1}=g_{k+2}.$ A contradiction.\newline
 1e) $s=5.$ Since the first $X$ is on positions $\{4k,...,4k+4\}$ and the second $X$ is on positions $\{4k+5,...,4k+9\},$ then $g_{4k+3}=g_{4k+8}$ and $g_{4k+4}=g_{4k+9},$ such that $g_{k+1}=g_{k+2}$ and $g_{k+1}=1-g_{k+2}.$ A contradiction.\newline
 1f) $s=6.$ Since here the first $X$ is on positions $\{4k,...,4k+5\},$ the second $X$ is on positions $\{4k+6,...,4k+11\}$ and the third $X$ is on positions $\{4k+12,...,4k+17\},$ then $g_{4k+1}=1-g_k=g_{4k+7}=g_{k+2},$ $g_{4k+2}=1-g_{k+1}=g_{4k+8}=g_{k+2}.$ Thus $g_k\neq g_{k+2}$ and $g_{k+1}\neq g_{k+2}$ thus $g_k=g_{k+1}.$ Further, $g_{4k+4}=g_{k+1}=g_k=g_{4k+10}=1-g_{k+3}$ and $g_{4k+6}=1-g_{k+2}=g_{4k+12}=g_{k+3}.$ So $g_{k+3}\neq g_{k+1}=g_k$ and $g_{k+3}\neq g_{k+2}.$ It is impossible, since, by $1a),$ $\{g_k, g_{k+1}, g_{k+2}\}$ contains $\{0,1\}.$\newline
 In general, let first $s$ be odd.\newline
 2a) Let $s=4m+1, \enskip m\geq2.$\newline
  Choose in the first $X$ $g_{4k+3}$ and in the second $X$ $g_{4k+3+s}.$ Then we have $g_{4k+3}=g_{4k+3+s}=g_{4k+4m+4}$ or $g_{k+1}=g_{k+m+1}.$ Now in the first $X$ we choose $g_{4k+4}$ and in the second $X$ $g_{4k+4+s}.$ Then we  have $g_{4k+4}=g_{4k+4+4m+1}$ or $g_{k+1}=1-g_{k+m+1}.$ So $g_{k+m+1}= 1-g_{k+m+1},$ a contradiction.\newline
 2b) Let $s=4m+3, \enskip m\geq1.$ \newline
  In the same way, choosing $g_{4k}=g_k$ and $g_{4k+1}=1-g_k$ in the first $X$ and, comparing with $g_{4k+4m+3}=g_{k+m+1},$ $g_{4k+4m+4}=g_{k+m+1}$ in the second $X,$ we obtain $g_{k+m+1}=g_{k}$ and $g_{k+m+1}=1-g_{k}.$ A contradiction.\newline
  Now let $s$ be even. \newline
 3a) Let $s=4m+2, \enskip m\geq2.$ \newline
  We have the following 4 pairs of equations:
  \newpage
  $g_{4k+1} = 1-g_{k},\enskip g_{4k+4m+3} = g_{k+m+1};$\newline
  $g_{4k+2} = 1-g_{k+1},\enskip g_{4k+4m+4)}= g_{k+m+1};$\newline
  $g_{4k+4} = g_{k+1},\enskip g_{4k+4m+6} = 1-g_{k+m+2};$\newline
  $g_{4k+6} = 1-g_{k+2}, g_{4k+4m+8} = g_{k+m+2}.$\newline
From the first two pairs we find $g_{k}=g_{k+1}.$ From the second two pars we find $g_{k+1}=g_{k+2}.$ So $g_{k}=g_{k+1}=g_{k+2},$ a contradiction.\newline
 3b) Let $s=4m,  \enskip m\geq2$ such that $m$ has a form considered in $1)-3a).$\newline
  Taking into account that the cases when the first element of $X$ is $4k+1,2,3$
 are proved in the same way, we can consider all of them as being proven in cases $1)-3a).$ Then this last
  case is proved by an induction. Indeed, from the proof of the cases $1)-3a)$ we conclude
  that if $s>6,$ then already the first $X$ does not coincide with the second $X.$
  Now we have $g_{4k}=g_k=g_{4k+4m}=g_{k+m},$ $g_{4k+3}=g_{k+1}=g_{4k+3+4m}=g_{k+m+1},$
  further,
  $$ g_{4k+4i-1+4m}=g_{k+i+m},\enskip i=1,2,...,m-1.$$
  So, we find
  $$ g_k=g_{k+m}, g_{k+1}=g_{k+m+1},...,g_{k+m-1}=g_{k+2m-1}$$
  which in case $s=m$ contradicts the impossibility of coincidence of the first and the second
  $X$ in the cases of $m$ in $1)-3a).$\newline
 3c) Finally, if $m=4t.$ Suppose, using induction, $s=m$ is impossible. Then
  $$ g_k=g_{k+m}, g_{k+1}=g_{k+m+1},...,g_{k+m-1}=g_{k+2m-1}$$
  is impossible. Then
  $$ g_{4k}=g_{4k+4m}, g_{4k+4}=g_{4k+4+4m},...,g_{4k+4m-4}=g_{4k+4m-4+4m}$$
  is also impossible. This means that $s=4m$ is impossible.
\end{proof}
\section{Proof of Theorem 6}
 \begin{proof} Suppose the sequence $R$ contains a subsequence of the form $XXXXX$ with the positions
 $$(4k,...,4k+s-1)(4k+s,...,4k+2s-1)...(4k+4s,...,4k+5s-1),$$
 where $k\geq0, s\geq1$ are integers.
 Using formulas of Theorem \ref{t1}, we obtain the proof in an analogous way as the proof of Theorem \ref {t5}, distinguishing the following cases.
 1a)  $s=1.$ Since $r_{4k}=r_k$ and $r_{4k+1}=1-r_k$ and also  $r_{4k+4}=r_{k+1}$ and $r_{4k+5}=1-r_{k+1,}$ then $R$ contains no five equal terms. \newline
 1b) $s=2.$ We have $r_{4k}=r_{k}= r_{4k+2}=r_{2k+1}.$ Then $r_{2k+1}=r_{k},$ so $k$ is odd.
  On the other hand,  $r_{4k+1}=1-r_{k}= r_{4k+3}=r_{2k+1}.$ then $r_{2k+1}=1-r_{k},$ so $k$ is even. A contradiction.
 \newline
 1c) $s=3.$ We have $r_{4k}=r_{k}=r_{4k+3}=r_{2k+1}.$ Then $k$ is odd. On the
  \newpage
  other hand, $r_{4k+1}=1-r_{k}=r_{4k+4}=r_{k+1}=r_{4k+7}=r_{2(k+1)+1},$ then $r_{2(k+1)+1}=r_{k+1}.$ So $k+1$ is odd. A contradiction.\newline
 1d) $s=4.$ We have $r_{4k}=r_{4k+4}=r_{4k+8}=r_{4k+12}=r_{4k+16},$ then $r_k=r_{k+1}=r_{k+2}r_{k+3}=r_{k+4}.$ This contradicts 1a).\newline
 1e1) $s=5,$ $k$ is odd. We have  $r_{4k}=r_{k}=r_{4k+5}=1-r_{k+1}$ and $r_{4k+1}=1-r_{k}=r_{4k+6}=r_{2(k+1)+1}=1-r_{k+1}.$ Thus, $r_{k}=1-r_{k+1}$ and $r_{k}=r_{k+1}.$ A contradiction.\newline
 1e2) $s=5, k$ is even. We have $r_{4k+4}=r_{k+1}=r_{4k+9}=1-r_{k+2}$ and
 $r_{4k+5}=1-r_{k+1}=r_{4k+10}=r_{2(k+2)+1}=1-r_{k+2}.$ So $r_{k+1}=1-r_{k+2}$  and $r_{k+1}=r_{k+2}.$  A contradiction.\newline
 1f) $s=6.$ $r_{4k}=r_k=r_{4k+6}=r_{2(k+1)+1}$ and $r_{4k+1}=1-r_{k}=r_{4k+7}=r_{2(k+1)+1}.$
A contradiction.\newline
In general, let first $s$ be odd.\newline
2a) Let $s=4m+1, k$ be odd. We have $r_{4k+1}=1-r_k=r_{4k+4m+2}=r_{2k+2m+1};$
$r_{4k+2}=r_{2k+1}=r_k=r_{4k+4m+3}=r_{2k+2m+1}.$ A contradiction.\newline
2b) Let  $s=4m+1, k$ be even. We have $r_{4k+5}=1-r_{k+1}=r_{4k+4m+6}=r_{2(k+m+1)+1};$
$r_{4k+6}=r_{2(k+1)+1}=r_{k+1}=r_{4k+4m+7}=r_{2(k+2m+1)+1}.$ A contradiction.\newline
2c) Let $s=4m+3, k$ be odd. We have $r_{4k+5}=1-r_{k+1}=r_{4k+4m+8}=r_{k+m+2};$
$r_{4k+6}=r_{2(k+1)+1}=1-r_{k+1}=r_{4k+4m+9}=1-r_{k+m+2}.$ A contradiction.\newline
2d) Let $s=4m+3, k$ be even. We have $r_{4k+1}=1-r_{k}=r_{4k+4m+4}=r_{k+m+1};$
$r_{4k+2}=r_{2k+1}=1-r_{k}=r_{4k+4m+5}=1-r_{k+m+1}.$ A contradiction.\newline
3) Let $s=4m+2.$ We have $r_{4k}=r_k=r_{4k+4m+2}=r_{2(k+m)+1}$ and $r_{4k+1}=1-r_{k}=r_{4k+4m+3}=r_{2(k+m)+1}.$ A contradiction. \newline
4) Let $s=4m.$  This case is considered in the same way as in the proof of Theorem \ref{t5}.
\end{proof}
\section{ Transcendency of numbers $0.G$ and $0.R$}
Together with the Thue-Morse constant $0.T=0.01101001100101..._2$ which
 is given by the concatenated digits of the Thue-Morse sequence A010060
 and interpreted as a binary number, consider the constant $0.G=0.0101101001..._2$ and the constant $0.R=0.01111011100..._2$
 which are given by the concatenated digits of the sequences  $G=A269027$ and $R=A268411$ respectively and interpreted as binary numbers. Mahler \cite{7} proved that $0.T$ is a
 transcendent number. Now we have a possibility to show that $0.G$ and $0.R$ are both transcendental numbers. Allouche (private communication) noted that both sequences $G$ and $R$
 are 2-automatic. Then numbers $0.G$ and $0.R$ also are called 2-automatic. Formally 2-automatic numbers could be 
 \newpage
 rational. But in view of Theorems \ref{t5} and \ref{t6}, the numbers $0.G$ and $0.R$ cannot be eventually periodic and, hence, cannot be rational. But in 2007 Adamczewski and Bugeaud \cite{1} obtained a remarkable result: all irrational automatic numbers are transcendental. So, in particular, numbers $0.G$ and $0.R$ are transcendental.
\newline

 \section{Several interesting author's problems}

 The following four problems are still unsolved (except for $B)$).

A) \cite{8}. For which positive numbers $a,b,c,$ for every nonnegative $n$ there
exists $x\in\{a,b,c\}$ such that $t_{n+x}=t_n?$ \newline
This problem was solved in \cite{8} only partially. For example, for every
$a\geq1,\enskip k\geq0,$ the triple $\{a, a+2^k, a+2^{k+1}\}$ is suitable.
However, there are other infinitely many solutions.\newline

B) Conjecture \cite{8}. Let $u(n)=(-1)^{t_n}|_{n\geq0}$ and $a$ be a positive
 integer. Let $\{l_0<l_1<l_2<...\}, \enskip \{m_0<m_1<m_2<...\}$ be
 integer sequences for which  $u(l_i+a)=-u(l_i)$, $u(m_i+a)=u(m_i).$
Let $\beta_a(n)=u(l_n), \enskip\gamma_a(n)=u(m_n).$
Then the sequences $\beta_a, \gamma_a$ are periodic, of the
 smallest period $2^{v(a)+1},$ where $v(a)$ is such that $2^v(a)||a.$ They satisfy
 $\beta_a=-\gamma_a.$\newline
 This conjecture was proved by Allouche \cite{2}.\newline

 C) (in A268866 \cite{10}). Let $v(n)$ be the maximal number $k$ such that
 $g_r=t_{r+n},\enskip r=0,1...,k-1 ($ if $k=0,$ there is no equality already for $r=0.)$ Let $\{a(n)\}$ be the sequence of records
 in the sequence $\{v(n)\}.$ Conjecture: 1) Let $l(n)$ be the position in
 $\{v(n)\}$ corresponding to $a(n).$ Then $l(n)=(2/3)(4^n-1),$ if $n$ is even, and $l(n)=(2/3)(4^{n-1}-1)+3\cdot4^{n-1},$ if $n$ is odd; 2) $a(n)=2l(n)+2,$ if $n$
 is even, and $a(n)=(7l(n)+12)/11,$ if $n$ is odd.\newline

 D) (A dual problem: in A269341 \cite{10}). Let $w(n)$ be the maximal number $k$ such that
 $g_r=1-t_{r+n},\enskip r=0,1...,k-1.$ Let $\{b(n)\}$ be the sequence of records
 in the sequence $\{w(n)\}.$ Denote by $m(n)$ the position in  $\{w(n)\}$
 corresponding to $b(n).$ Then $m(0)=0, m(1)=1.$ Conjecture: 1) for even
 $n\geq2,\enskip m(n)=(2/3)(4^{n-1}-1);$ for odd $n\geq3,\enskip m(n)=(2/3)(4^{n-2}-1)+3\cdot4^{n-2};$ 2) for even $n\geq2,\enskip b(n)=2m(n)+2;$
  for odd $n\geq3,\enskip b(n)=(7m(n)+12)/11.$\newline
 The author hopes that this paper will help to solve at least the problems $C)$ and $D).$
The paper is connected with the following sequences in \cite{10}: A000695, A010060,
A039724, A069010, A020985, A022155, A203463, A268382, A268383, A268411, A268412,
A268415, A268865, A268866, A268272,
\newpage
 A268273, A268476, A268477, A268483, A269003,
A269027, A269340, A269341, A269458, A269528, A269529.
\newline
\section{Acknowledgement}
The author is grateful to Jean-Paul Allouche for indication of article \cite{1} and a useful discussion of it. He also thanks Peter J. C. Moses whose numerical results were very helpful.

\end{document}